\documentclass[12pt]{article}
\usepackage{amsmath, amsthm,  amsfonts, amssymb, fdsymbol}
\usepackage{graphicx}
\usepackage{color}
\usepackage[colorlinks]{hyperref}
\usepackage[all]{xy}
\usepackage{enumitem}

\usepackage{subfigure}
\usepackage{tikz}
\usetikzlibrary[arrows,shapes,graphs,decorations.pathmorphing,backgrounds,positioning,fit,petri]

\newcommand{\op}{\oplus}
\newcommand{\od}{\odot}

\newdimen\xleftright
\xleftright=\textwidth
\advance \xleftright by -10.5cm

\newtheorem{theorem}{Theorem}[section]
\newtheorem{lemma}[theorem]{Lemma}
\newtheorem{proposition}[theorem]{Proposition}
\newtheorem{corollary}[theorem]{Corollary}
\theoremstyle{definition}
\newtheorem{definition}[theorem]{Definition}
\newtheorem{example}[theorem]{Example}

\theoremstyle{remark}
\newtheorem{remark}[theorem]{Remark}
\numberwithin{equation}{section}

\begin{document}

\title{\bf A Complete Axiomatisation for the Logic of Lattice Effect Algebras}
\author{Soroush Rafiee Rad, Amir Hossein Sharafi, Sonja Smets}
\date{}	
	\maketitle

\begin{abstract}
In a recent work Foulis and Pulmannov\' a \cite{Foulis2012} studied the logical connectives in lattice effect algebras. In this paper we extend their study and investigate further the logical calculus for which the lattice effect algebras can serve as semantic models. We shall first focus on some properties of lattice effect algebras and will then give a complete axiomatisation of this logic.
\end{abstract}

\textbf{Keywords:} lattice effect algebra, weak lattice effect algebra, unsharp quantum logic, Sasaki arrow.

\section{Introduction}

Hilbert spaces have been shown to provide a suitable mathematical setting for the modelling and analysis of quantum systems, where, in particular, the set of tests (and more generally measurements) on the system are identified with the set of closed linear subspaces of the Hilbert space, or equivalently the set of projectors on these closed linear subspaces. It is well known that given a Hilbert space $\mathcal{H}$\footnote{of dimension at least 3} and a set of tests $\Pi(\mathcal{H})$ (capturing the properties of some quantum system) every probability function on $\Pi(\mathcal{H})$ can be defined from a density operator. The set of projectors, however, does not capture the complete set of operators on $\mathcal{H}$ that are assigned a probability by the probability functions defined from the density operators. The notion of {\em effect} on a Hilbert space is defined to capture all such operators, be it a projector or not. In this sense effects of Hilbert spaces can be thought as capturing a generalised notion of tests and measurements on quantum systems.\\
These generalised tests are particularly relevant in the study of the unsharp approach to quantum systems \cite{Dalla2004}. In the sharp approach there is no uncertainty concerning the properties of the quantum systems that are measured. In this sense the uncertainty involved in a measurement (represented by the assignment of probabilities other than 0 or 1) can be understood as the uncertainty concerning the state of the system and points to the fact that ``even a pure state in QT does not represent logically complete information that is able to decide any possible physical event" \cite{Dalla2004}. The unsharp approach allows for an extra layer of uncertainty concerning the properties themselves, which can be understood as the uncertainty arising from the accuracy of the measurement apparatus or the operational definitions of the physical quantities captured by the property that is being measured. \\ 
The notion of effect and the algebraic structures for their representations have been investigated extensively and there is a rich literature on the study of both the effect algebras and the lattice effect algebras, developed in the works of Ludwig, Kraus, Davies, Busch, Lahti, Mittelstaedt, Gudder, Foulis, Bennett, Dalla Chiara, Giuntini, Pulmannov\' a and Dvure\v censkij,  amongst others,
 \cite{Bennett1995},\cite{Busch1991}, \cite{Dalla2004}, \cite{Davies1975}, \cite{Dvurecenskij2000}, \cite{Foulis1994}, \cite{Greechie1994}, \cite{Gudder1995}, \cite{Gudder1998}, \cite{Jenca2007}, \cite{Jenca1999}, \cite{Kraus1983},  \cite{Ludwig1983}, \cite{Riecanova2000}, \cite{Riecanova2009}. 
Our goal in this paper is to contribute to this literature and is in-line with some recent works on the study of the logical structure of the lattice effect algebras, for example, \cite{Chajda2017}.\\
Following Foulis and Pulmannov\' a \cite{Foulis2012}, we are in particular interested to study the lattice effect algebras as a model for a logical system in the same manner that orthomodular lattices and MV-algebras are studied as the models for sharp quantum logic and Lukasiewicz many-valued logic respectively. We will thus study a logical calculus LEL for which the lattice effect algebras can serve as a semantic model. A lattice effect algebra carries natural operations of orthosupplementation, Sasaki product and Sasaki arrow which, following \cite{Foulis2012}, we shall consider as representing the operations of negation, conjunction and implication in our logical calculus LEL. In this analysis we extend the work of Foulis and Pulmannov\' a by looking beyond the algebraic structure and the partial order of the lattice effect algebras and will investigate the underlying logic in terms of proof theory and soundness and completeness.\\
After giving the preliminary definitions and constructions in Section \ref{sec2}, we first study the properties of lattice effect algebras in Section \ref{sec3} before proceeding to analyse the underlying logic. In particular, we shall first characterize the residuation law of Sasaki product and Sasaki arrow in a manner analogous to the characterization of the duality between Sasaki projection and Sasaki hook on ortholattices in terms of the orthomodular law. Next we shall give a canonical construction of lattice effect algebras from  a conjunction/implication lattice and then proceed to study the underlying logic of the lattice effect algebras and  will give a sound and complete axiomatization of this logic in Section \ref{sec4}. 

\section{Preliminaries} \label{sec2}

Let $\mathcal{H}$ be a Hilbert space and $\Pi(\mathcal{H})$ and $\mathcal{D}(\mathcal{H})$ be the set of projectors and the set of density operators on $\mathcal{H}$ respectively. Every density operator $\rho$ defines a  probability measure, $m_{\rho}$ on $\Pi(\mathcal{H})$ where for each $\mathcal{P} \in \Pi(\mathcal{H})$, $m_{\rho}(\mathcal{P})$ captures the probability of event $\mathcal{P}$ (or more precisely the probability of a successful test of the property represented by $\mathcal{P}$) at the state represented by $\rho$.\\
As mentioned in \cite{Dalla2004}, an interesting question arises as to the adequacy of this mathematical representation. On one direction, given a set of events $\Pi(\mathcal{H})$ for some $\mathcal{H}$ with dimension at least 3, it is well known by the Gleason’s Theorem, that the set of density operators $\mathcal{D}(\mathcal{H})$ gives an optimal representation of the state space in the sense that every probability function on the set of events will be defined by a density operator in $\mathcal{D}(\mathcal{H})$. In the other direction, however, given the set of density operators $\mathcal{D}(\mathcal{H})$, the set of projectors of $\mathcal{H}$ is not the largest set of operators that are assigned a probability according to the Born rule by the probability measures defined by the density operators in $\mathcal{D}(\mathcal{H})$. In particular, there are bounded linear operators $E$ on $\mathcal{H}$ that resemble events (in the sense that they are assigned probability values by density operators) which are not idempotent and hence are not included in the set of projectors on $\mathcal{H}$. The notion of {\em effect} on Hilbert spaces is defined to capture all such operators whether they are given by a projector or not and as mentioned above, in this sense they generalise the notion of tests for a quantum system.\\
An important characteristic feature of quantum systems is that not all such generalised tests can be performed simultaneously. A natural operation to consider on the set of effects is hence that of an {\em exclusive} disjunction for the incompatible tests. The set of effects of a Hilbert space with this operation forms an algebra that was first introduced by Foulis and Bennett \cite{Foulis1994} and is referred to as an {\em effect algebra}. The algebra of effects generalizes orthomodular lattices and MV-algebras which include non-compatible pairs of tests and  un-sharp tests respectively. Thus by generalising to effect algebras one can assign probabilities to events that represent properties that may be incompatible as well as those with fuzziness, uncertainty or un-sharpness. The operation of the algebra, as we shall shortly describe in more details, can be used to define a partial order on the set of effects. When the set of effects with this partial order forms a lattice the resulting structure is called the lattice-ordered effect algebra or {\em lattice effect algebra} for short.\\
\noindent We will start this section by recalling some basic notions and definitions related to effect algebras and the partial order imposed by their main operation. We shall briefly look at some derivative operations that can be defined on these structures which will be useful later on and we will point out some useful properties of these algebras. \\

\begin{definition}
A \textit{bounded involutive lattice} is a structure $(L;\leq,',0,1)$ where:
\begin{itemize}
\item[(i)] $(L;\leq,0,1)$ is a bounded lattice consisting of a set of elements $L$ a binary relation $\leq$ and two special elements $0$ and $1$.
\item[(ii)] $'$ is a unary operation (called involution) that satisfies the following conditions:
\begin{itemize}
\item[(a)] for all $a \in L$, $a=a''$  \hspace{10mm} (b) For all $a, b \in L$, if $a \leq b$ then $b'\leq a'$
\end{itemize}
\end{itemize}
\end{definition}

\noindent Given a bounded lattice $(L;\leq,0,1)$, the order induces the least upper bounds and greatest lower bounds of its elements, denoted as standard by $\wedge$ and $\vee$ respectively. An element $a$ of $L$ is {\it sharp} if $a\wedge a'=0$. A bounded involutive lattice $L$ is an {\it ortholattice} if every element of $L$ is sharp.

\begin{definition}
An ortholattice $(L;\leq,0,1)$ is called an {\it orthomodular lattice} if for every $a,b \in L$
$$a\leq b \, \Rightarrow\, a\vee(a'\wedge b)=b.$$
\end{definition}

\begin{definition}\cite{Foulis2012}
A conjunction/implication lattice (CI-lattice for short) is a system $(L;\leq,\cdot,\rightarrow,0,1)$ in which $(L;\leq,0,1)$ is a bounded lattice  and for all $a,b,c \in L$ the two operations $\cdot$ and $\to$ satisfy the following conditions:
\begin{itemize}
\addtolength{\itemindent}{5mm}
\item[(CI1)] $1\cdot a=a\cdot 1=a$; \hspace{15mm}(CI2) $a\cdot c\leq b$ if and only if $c\leq a\rightarrow b$.
\end{itemize}
\end{definition}

\noindent Every CI-lattice satisfies the following properties:

\begin{itemize}
\addtolength{\itemindent}{5mm}
\item[(CI3)]{\makebox[5cm] {$a \cdot b\leq a$;\hfill} (CI4) $a \cdot (a\rightarrow b)\leq b$;}
\item[(CI5)]{\makebox[5cm] {$a\leq b\ \Leftrightarrow\ a\rightarrow b=1$; \hfill}(CI6) $ a\rightarrow(b\wedge c)=(a\rightarrow b)\wedge (a\rightarrow c)$.}
\end{itemize}

\noindent Let $L=(L;\leq,\cdot,\rightarrow,0,1)$ be a CI-lattice and for all $a \in L$ define $a'= a \to 0$. Then $L$ is called an {\it involutive CI-lattice} if $(L;\leq,',0,1)$ is a bounded involutive lattice.

\begin{definition}
In the lattice $(L;\leq)$, a map $A: L \to L$ is called residuated if there exists a map $B: L \to L$ such that $(Ax \leq y)$ iff $(x \leq By)$. $B$ is called the residual map of $A$.
\end{definition}


\begin{definition}\cite{Foulis1994}
An {\it effect algebra} is a structure $(E; \op, 0, 1)$ consisting of a set $E$, two special elements 0 and 1, and a partially defined
binary operation $\op$ on $E\times E$ satisfying the following conditions for every $a, b, c\in E$
\begin{itemize}
\addtolength{\itemindent}{3mm}
\item[(E1)]  if $a\op b$ is defined, then $b\op a$ is defined and $a\op b=b\op a$;
\item[(E2)] if $b\op c$ and $a\op(b\op c)$ are defined, then $a\op b$ and $(a\op b)\op c$ are defined and $a\op(b\op c)=(a\op b)\op c$;
\item[(E3)] for every $a\in E$, there exists a unique $a'\in E$ such that $a\op a'$ is defined and $a\op a'=1$; we call the element $a'$ the {\it orthosupplement} of $a$;
\item[(E4)] if $a\op 1$ is defined, then $a=0$.
\end{itemize}
If there is no ambiguity we will simply write $E$ to denote the effect algebra $(E; \op, 0, 1)$. For $a, b \in E$, when $a\op b$ exists, we say that $a$ is orthogonal to $b$ and will denote this by $a\perp b$.
\end{definition}

\noindent On an effect algebra $E$ one can define a partial order $\leq$ as
$$a \leq b \iff \exists c\in E \left(a \op c = b \right).$$ 

\noindent This will make $E$ into a partially ordered set and if $(E; \leq)$ is a lattice, then $(E; \op, 0, 1)$ is called a \textit{lattice effect algebra} (LEA for short). Below we list a useful set of properties of effect algebras, see  \cite{Dvurecenskij2000} or \cite{Foulis1994} for proofs.\\

\begin{proposition}

Let $E$ be an effect algebra and $a, b$ and $c$ be elements of $E$:
\begin{align*}
&(e1) \,\,\, a''=a; & &(e2)\,\,\, 1'=0, 0'=1;  \nonumber \\
&(e3)\,\,\, a\perp 0, a\op 0=a; & &(e4)\,\,\, a \perp 1 \Leftrightarrow a=0; \nonumber\\
&(e5)\,\,\, a \op b=0 \Leftrightarrow a=b=0; & &(e6)\,\,\, a \perp b \Leftrightarrow a \leq b'; \nonumber \\
&(e7)\,\,\, a \leq b\ \Rightarrow b' \leq a'; & &(e8)\,\,\, a\op c = b\op c \Rightarrow a = b; \nonumber \\
&(e9)\,\,\, a\op c \leq b\op c \Rightarrow a \leq b; & &(e10)\, a \leq b \Rightarrow a\perp(a\op b')'; \nonumber\\
&(e11)\, a \leq b \Rightarrow \ a\op(a\op b')'=b;  & & (e12)\, a\op b=c \Leftrightarrow a'=b\op c'; \nonumber \\
&(e13)\, a\op b=c\Leftrightarrow a=(b\op c')'. & & \nonumber
\end{align*}
\end{proposition}

\noindent Let $a,b$ be two elements of an effect algebra with $a \leq b$, the element $c$ for which $a\op c=b$ is uniquely determined by the property (e8), and we define $b \ominus a=c$. Moreover, from (e1) and (e7), it will be easy to see that every lattice effect algebra is a bounded involutive CI-lattice in which the involution of every element is its orthosupplement.\\

\begin{definition}
Let $E$ be an effect algebra. The partial operation $\od$ on $E$ is defined  as
 $$a\od b=(a'\op b')'.$$ 
\end{definition}

\begin{proposition} Let $E$ be an effect algebra. Then for all $a,b, c \in E$,

\begin{itemize}
\addtolength{\itemindent}{5mm}
\item[{\rm(e13)}]  If $a\od b$ is defined, then $b\od a$ is defined and $a\od b=b\od a$;
\item[{\rm(e14)}] $a\od a'=0$;
\item[{\rm(e15)}] If $a\od 0$ is defined, then $a=1$;
\item[{\rm(e16)}] $a\od b$ is defined if and only if $a'\leq b$;
\item[{\rm(e17)}] If $a\od c=b\od c$, then $a=b$;
\item[{\rm(e18)}] $a\od a=a$ if and only if $a=1$.
\item[{\rm(e19)}] If $b\od c$ and $a\od(b\od c)$ are defined, then $a\od b$ and $(a\od b)\od c$ are defined and $a\od(b\od c)=(a\od b)\od c$;
\end{itemize}
\end{proposition}

\begin{definition}\cite[Definition 3.4]{Bennett1995}
Let $L$ be an LEA, then for $p, q \in L$, the Sasaki mapping $\varphi: L \times L \to L$ is defined by $\varphi(p, q) := p \ominus (p \wedge q') =
(p' \op (p \wedge q'))'= p\od(p'\vee q)$.

\end{definition}

\begin{definition} \cite{Foulis2012} Let $(E,\leq, 0, 1)$ be a lattice effect algebra. The binary operations $\otimes: E \times E \to E$ and $\to_{s}: E \times E \to E$ defined from Sasaki mapping as  
$$a\otimes b=\varphi(a,b)= (a' \op (a \wedge b'))'$$
and
$$a \to_{s} b = (a \otimes b')'=a'\op(a\wedge b)$$
are total operations on $E$ which we shall refer to as the Sasaki product and the Sasaki arrow respectively.
 \end{definition}

\begin{proposition}\label{LEAtoCI}(From lattice effect algebras to CI-lattices)
Let $(E;\op, 0,1)$ be a lattice effect algebra and define binary operation $\otimes$ and $\to_s$ by Sasaki product and Sasaki arrow as above. Then 
$$(E; \leq, ', \otimes, \to_s, 0,1 )$$
is an involutive CI-lattice which satisfies the following properties:\\
{\rm(i)}  $c\leq a,\ c\leq b\ \Rightarrow c\leq a\otimes(a\rightarrow_s b)\quad \text{\small(divisibility law)}$,\\
{\rm(ii)}  $[(a\rightarrow_s b)\rightarrow_s c]\otimes[(b\rightarrow_s a)\rightarrow_s c]\leq c\quad \text{\small(strong prelinearity law)}$,\\
{\rm(iii)}  $a\otimes b\leq c'\ \Rightarrow\ a\otimes c\leq b'\quad \text{\small(self-adjointness law)}$.
\end{proposition}
\noindent See \cite{Foulis2012} for details. 

\begin{proposition} The Sasaki product is a residuated map, and its residual map is the Sasaki arrow.
\end{proposition}

\begin{proof}
This can be easily verified from the self-adjointness law and fact that $'$ is an involution.
\end{proof}

\noindent Sasaki arrow has been shown to provide a suitable notion of implication in a lattice effect algebra. We finish this section by pointing to some properties of Sasaki arrow on lattice effect algebras.

\begin{definition}\cite{Borzooei2018}
Let $(L, \leq)$ be a bounded lattice. A partial binary operation $\Delta$ on $L$ is called a {\it partial t-norm} if it satisfies the following conditions:
\begin{itemize}
\addtolength{\itemindent}{3mm}
\item[(i)] $1 \, \Delta \, a=a$;
\item[(ii)] if $a\, \Delta \, b$ is defined, then $b \,\Delta \, a$ is also defined and $a \, \Delta \, b=b \,\Delta\, a$;
\item[(iii)] if $b\, \Delta \, c$ and $a \, \Delta\, (b\, \Delta \,c)$ are defined, then $a\, \Delta \, b$ and $(a\, \Delta \, b) \, \Delta \, c$ are also defined and $a\, \Delta \,(b\,\Delta\, c)=(a\, \Delta\, b)\, \Delta\, c$;
\item[(iv)] if $a \leq b$, $c \leq d$ and $a\, \Delta\, c$, $b \,\Delta\, d$ are defined then $a \,\Delta\, c\leq b\, \Delta\, d$.
\end{itemize}
\end{definition}
\noindent Notice that the partial binary operation $\od$ defined above is a partial t-norm.
\begin{definition}\cite{Borzooei2018}
Let $L$ be a bounded involutive lattice. We say that a partial binary operation $\rightarrow$ is a \textit{weak partial t-implication} or a \textit{weak pt-implication} if there is a partial t-norm $\Delta$ on $L$ such that it satisfies the following conditions:
\vspace{4mm}

$
\begin{array}{ll}
(E)&a\leq b\ \ \text{iff}\ \ a\rightarrow b=1;\\
(MP^{pt})&a\Delta(a\rightarrow b)\leq b\ \text{when}\ a\Delta(a\rightarrow b)\ \text{is defined};\\
(MT^{pt})& b'\Delta(a\rightarrow b)\leq a'\ \text{when}\ b'\Delta(a\rightarrow b)\ \text{is defined};\\
(NG^{pt})& a\Delta b'\leq(a\rightarrow b)'\ \text{when}\ a\Delta b'\ \text{is defined}.
\end{array}
$
\vspace{4mm}

\noindent A weak pt-implication is a called a \textit{pt-implication} if there is a binary operation $\divideontimes$ on $L$ satisfying:
\vspace{3mm}\\
\hspace*{1cm}$(R)\quad a\divideontimes c\leq b\ \ \text{iff}\ \ c\leq a\rightarrow b$ 
\end{definition}

\begin{definition}
In a lattice effect algebra $E$, $a,b\in E$ are called \textit{compatible} (denoted by $a\leftrightarrow b$) if $a\vee b=a\op(b\ominus(a\wedge b))$. The \textit{compatible center} of $E$ is defined as the set $B(E)=\{x\in E\mid x\leftrightarrow y\ \text{for all}\ y\in E\}$ and the set of \textit{central} elements of $E$ is defined as $C(E)=B(E)\cap S(E)$ where 
$S(E)=\{z\in E\mid z\wedge z'=0\}$ is the set of all sharp elements of $E$.
\end{definition}

\begin{theorem}{\rm\cite{Borzooei2018}}
Let $E$ be a lattice effect algebra, $\rightarrow$ be a pt-implication on $E$. Then $a\rightarrow b=a'\vee b$, for all $a,b\in C(E)$, if and only if $\rightarrow$ is the Sasaki arrow.
\end{theorem}

\begin{definition}
We say that an arrow operation $\Rightarrow$ on a lattice $L$ is \textit{stricter} than another arrow operation $\rightarrow$ on $L$ if $a\Rightarrow b\leq a\rightarrow b$ for all $a,b\in L$.
\end{definition}

\begin{theorem}{\rm\cite{Borzooei2018}}
Let $E$ be a lattice effect algebra. Then any pt-implication $\rightarrow$ on $E$ is stricter than $\rightarrow_s$.
\end{theorem}

%

\section{Some properties of Lattice Effect Algebra}\label{sec3}

\noindent We will now introduce a weakening of the notion of lattice effect algebras. This will allow us to investigate the duality between the Sasaki product and the Sasaki arrow defined above and to look into which properties of the lattice effect algebras are characterised by this duality. Our result in this section are along the same lines as the classical result showing the equivalence of, on the one hand the duality between Sasaki hook and Sasaki projection and on the other hand the orthomodular law on ortholattices. 

\begin{definition}
A \textit{weak lattice effect algebra} is a bounded involutive lattice $(L; \leq,\op,',0,1)$ in which $a\op b$ is defined if and only if $a\leq b'$ and satisfies the following properties:
\begin{itemize}
\addtolength{\itemindent}{5mm}
\item[(W1)] If $a \op b$ is defined then  $a\op b=b\op a$;
\item[(W2)] If $b \op c$ and $a \op (b\op c)$ are defined then so are $a \op b$ and  $(a\op b)\op c$, and we have $a\op(b\op c)=(a\op b)\op c$;
\item[(W3)] For every $a\in L$, $a\op 0=a$.
\end{itemize}
\end{definition}
In a weak lattice effect algebra $L$, we write $a \bot b$ when $a \op b$ is defined. We now give a couple of examples of weak lattice effect algebras, which are not lattice effect algebras.

\begin{example}
Let $L=\{0,a,b,a',b'\}$ and define a partial operation $\op$ such that $a\op b=a\op a'=b\op b'=1$. Then $(L; \leq,\op,',0,1)$ is a weak lattice effect algebra and the induced order is visualized in Fig 1-a.
\end{example}

\begin{example}
Let $L=\{0,a,b,c,a',b',c'\}$ and define a partial operation $\op$ such that $a\op a'=c$, $a\op b=b\op b'=c\op c'=1$, $a\op c'=a$, $a'\op c'=a'$ and $c'\op c'=c$. Then $(L; \leq,\op,',0,1)$ is a weak lattice effect algebra and the induced order is visualized in Fig 1-b.
\end{example}

\begin{figure}[h]
\begin{center}
\begin{tikzpicture}
 \matrix[row sep=5mm,column sep=5mm] {
  & \node (1) {$1$}; & \\
  \node (b1) {$b'$}; & & \node (a1) {$a'$};\\
  \node (a) {$a$}; & & \node (b) {$b$};\\
  & \node (0) {$0$}; & \\
};
\graph {
  (b1)--(1); (a1) -- (1);
  (a)--(b1); (b)--(a1);
  (0)--(a); (0)--(b);
};
\end{tikzpicture} Fig. 1-a.
\hspace{30mm}
\begin{tikzpicture}
 \matrix[row sep=5mm,column sep=5mm] {
  & \node (1) {$1$}; \\
  \node (b1) {$b'$}; & \node (c) {$c$};\\
  \node (a) {$a$}; & \node (a1) {$a'$};\\
  \node (c1) {$c'$}; & \node (b) {$b$};\\
  \node (0) {$0$}; & \\
};
\graph {
  (0)--(c1); (c1) -- (a); (a)--(b1);
  (0)--(b); (c1)--(a1); (a)--(c); (b1)--(1);
  (b)--(a1); (a1)--(c); (c)--(1);
};
\end{tikzpicture} Fig. 1-b.
\end{center}

\end{figure}

\begin{figure}[!hbtp]

 
\end{figure}

\begin{remark}
In proof of the following proposition and theorem we should notice that for a bounded involutive lattice we have $0'=1$.
\end{remark}

\begin{proposition}
Let $(L; \leq,\op,',0,1)$ be a weak lattice effect algebra. Then
\begin{itemize}
\addtolength{\itemindent}{5mm}
\item[{\rm(W4)}] If $a\perp 1$ then $a=0$;
\item[{\rm(W5)}] If $a\perp b$ and $a\op b=c$ then we have 
\begin{itemize}
\item[{\rm(i)}] $a, b\leq c$,
\item[{\rm(ii)}] $a\op c'\leq b'$ and $b\op c'\leq a'$;
\end{itemize}\item[{\rm(W6)}] If $a\op a'=r$ then $r'\leq a\leq r$;
\item[{\rm(W7)}] $a\leq b$ and $a\op b'\neq 1$ if and only if there is $c\in E$ such that $c \neq 0$ and $a\op c\leq b$;
\item[{\rm(W8)}] $a\op c\leq b$ if and only if $b'\op c\leq a'$.
\end{itemize}
\end{proposition}
\begin{proof}
\noindent {\bf(W4):} From the definition, $a\op 1$ is defined iff $a\leq 1'$ which means $a\leq 0$ and so $a=0$.\\

\noindent {\bf(W5):} Since $c\leq c$ we have $c\perp c'$ and so from $a\op b=c$ we get $(a\op b)\perp c'$. Hence from (W1) and (W2) we obtain $a\perp c'$, $b\perp c'$, $a\op c'\perp b$ and $b\op c'\perp a$ which imply the results.\\

\noindent {\bf(W6):} Using (W5), from $a\op a'=r$, we get $a,a'\leq r$ and therefore $r'\leq a\leq r$.\\

\noindent {\bf(W7):} Let $a\leq b$ and $a\op b'\neq 1$. Then there is $1\neq r\in L$ such that $a\op b'=r$ and so $a\op r'\leq b$ follows from (W5) in which $r'\neq 0$ is the desired $c$. Conversely assume that there is $0\neq c\in L$ such that $a\op c\leq b$. From (W5) we know $a\leq a\op c$ and $a\op b'\leq c'$. Hence using transitivity we get $a\leq b$ and since $c'\neq 1$ we get $a\op b'\neq 1$.\\

\noindent {\bf(W8):} Let $a\op c\leq b$. Then $a\op c\perp b'$ and from (W1) and (W2) we have $b'\op c\perp a$ which means $b'\op c\leq a'$. Similarly the converse is true.
\end{proof}

\begin{theorem}\label{RD}
Let $(L; \leq,\op,',0,1)$ be a weak lattice effect algebra. Then the following are equivalent:
\begin{itemize}
\addtolength{\itemindent}{5mm}
\item[{\rm(RD1)}] $(L;\op,0,1)$ is a lattice effect algebra where $a'$ is the orthosuplement of $a\in L$.
\item[{\rm(RD2)}] $a\otimes x\leq b$ if and only if $x\leq a\rightarrow_s b$.
\item[{\rm(RD3)}] $b'\leq a'\op (a\wedge x')$ if and only if $x\leq a'\op(a\wedge b)$.
\item[{\rm(RD4)}] $a\leq b$ if and only if $a\rightarrow_s b=1$.
\item[{\rm(RD5)}] For every $a\in L$, $a\op b=1$ if and only if $b=a'$.
\item[{\rm(RD6)}] If $a\op b=c$ then $a\op c'=b'$.
\item[{\rm(RD7)}] If $a\leq b$ then $a\op(a'\od b)=b$.
\item[{\rm(RD8)}] $a\leq b$ if and only if there is a unique $c\in L$ such that $a\op c=b$.
\end{itemize}
\end{theorem}
\begin{proof} 
\noindent (RD1)$\Rightarrow$(RD2) Follows from \cite{Foulis2012}.\\
\noindent (RD2)$\Rightarrow$(RD3) It is clear from definition of Sasaki product and Sasaki arrow and using the fact that $L$ is an involutive lattice.\\
\noindent (RD3)$\Rightarrow$(RD4) Let $a\leq b$. Then $b'\leq a'$ and by (W3) we have $b'\leq a'\op(a\wedge 0)$. Using (RD3) we have $1\leq a'\op(a\wedge b)$ and so $a\rightarrow_s b=1$. The converse follows from the backward argument.\\
\noindent (RD4)$\Rightarrow$(RD5) Let $a\in L$. Since $a\leq a$ then by (RD4) we have $a\rightarrow_s a=1$. Thus $a'\op(a\wedge a)=1$ which means $a\op a'=1$. On the other hand for $a\in L$ assume that $a\op b=1$. Then $a\perp b$ and we have $ b \leq a'$. Next $a' \to_{s} b=  a \op (a' \wedge b)= a \op b = 1$ by assumption. Thus by (RD4) we have $a' \leq b$. Thus we have $a' \leq b$ and $b \leq a'$ and hence $b= a'$.\\
\noindent (RD5)$\Rightarrow$(RD6) Let $a\op b=c$. Since from (RD5) we have $c\op c'=1$ and $(a\op b)\op c'=1$. By (W2), we have $a\op(b\op c')=1$ and using (RD5) again, it follows that $a\op c'=b'$.\\
\noindent (RD6)$\Rightarrow$(RD7) Let $a\leq b$. Then $a\op b'$ is defined and we have $a\op b'=a\op b'$. Thus using (RD6) we have $a\op(a\op b')'=b$ and thus $a\op(a'\od b)=b$.\\
\noindent (RD7)$\Rightarrow$(RD6) Let $a\op b=c$. Since $a\op b$ is defined then $a\leq b'$. Hence using (RD7) we obtain $a\op(a\op b)'=b'$ which means $a\op c'=b'$.\\
\noindent (RD7)$\Rightarrow$(RD8) Since (RD6) and (RD7) are equivalent, we freely use (RD6). Assume that there is $c\in L$ such that $a\op c=b$ then, from (RD7), we have $a\op b'=c'$ and so $a\op b'$ is defined which means $a\leq b$. Conversely, suppose that $a\leq b$ then we can get $c=a'\od b$. Moreover, using (RD7), we have $a\op b'=(a'\od b)'$ and so $c$ is unique element such that $a\op c=b$, for $a,b\in L$.\\
\noindent (RD8)$\Rightarrow$(RD1) Since $a\leq 1$, from (RD8), there is a unique $b$ such that $a\op b=1$. Moreover, if $a\op 1$ is defined then $a\leq 1'=0$ and so $a=0$. Therefore, $L$ satisfies all properties of a lattice effect algebra.
\end{proof}

\begin{remark}
The relation between lattice effect algebras and weak lattice effect algebras is analogues to that of orthomodular lattices and ortholattices. As the previous result shows, this is witnessed by the fact that adding the residuation law of Sasaki product and Sasaki arrow to a weak lattice effect algebra will make it into a lattice effect algebra.
\end{remark}

\begin{remark}
Obviously, if every element of a weak lattice effect algebra is sharp then it is an ortholattice. Furthermore, if we define the partial binary operation $\op$ on an ortholattice $(L; \leq,',0,1)$ by setting $a\op b=a\vee b$ for all $a,b \in L$ with $a\leq b'$, then $(L;\leq,\op,',0,1)$ is a weak lattice effect algebra in which every element is sharp and $a\op a'=1$, for all $a\in L$. Indeed, to show this, all one has to check is (W2). To see this let $b \perp c$ and $a\perp(b\op c)$. Then $b\leq c'$ and $a\leq(b\vee c)'$ so it is easy to see that $b\leq a'$ and $a\vee b\leq c'$. Thus $a\perp b$, $(a\op b)\perp c$ and $a\op(b\op c)=(a\op b)\op c=a\vee b\vee c$ as required.
\end{remark}

\noindent Proposition \ref{LEAtoCI} established how to make an involutive CI-lattice from a lattice effect algebra. In the next results we establish the other direction and show the conditions under which a lattice effect algebra can be constructed from an involutive CI-lattice.
\begin{theorem}\label{invlatweak}
Let $(L; \leq,\cdot,\rightarrow,',0,1)$ be a system in which $(L; \leq,',0,1)$ is a bounded involutive lattice and satisfies the following properties:
\begin{itemize}
\addtolength{\itemindent}{5mm}
\item[{\rm(cw1)}] $a\cdot b=(a\rightarrow b')'$;
\item[{\rm(cw2)}] If $a\leq b'$ then $a'\rightarrow b=b'\rightarrow a$ and $a\leq a'\rightarrow b$;
\item[{\rm(cw3)}] If $a\leq b'$ and $a\leq c'$ then $a'\rightarrow b\leq c'$ implies $a'\rightarrow c\leq b'$;
\item[{\rm(cw4)}] If $b'\leq c$ and $a'\leq b\cdot c$ then $a\cdot(b\cdot c)=(a\cdot b)\cdot c$;
\item[{\rm(cw5)}] $a\cdot 1=a$.
\end{itemize}
Then $L$ can be arranged into a weak lattice effect algebra $(L; \leq,\op,',0,1)$  by setting
$$a\op b=a'\rightarrow b \,\,\rm{ whenever} \,\, a\leq b' \,\,\rm{ and \,\, undefined\,\, otherwise.} \,\,$$
\end{theorem}
\begin{proof}
We need to show (W1), (W2) and (W3). First notice that if $a \op b$ is defined then by (cw1) $a \op b= (a' \cdot b')'$. For (W1) assume $a \op b$ is defined and thus by definition $a \leq b'$ then using (cw2)  we have $a'\rightarrow b=b'\rightarrow a$ and thus by definition $a \op b= b \op a$ . For (W2) assume $b \op c$ and $a \op (b \op c)$ are dfined, then by definition $b \leq c'$ (and $c \leq b'$) and $a \leq (b \op c)'$ (and $b \op c \leq a'$).  First we show that $a \op b$ and $(a \op b) \op c$ are defined. To see this notice that by the second part of (cw2) from $b\leq c'$ we have $b\leq b'\rightarrow c$, and $b\op c\leq a'$ is, by definition, $b'\rightarrow c\leq a'$. Thus by transitivity $b\leq a'$. Thus $a \leq b'$ and $a \op b$ is defiend.  Next, since $b \leq a'$ and also $b \leq c'$, by (cw3), $b'\rightarrow a\leq c'$. That is $b \op a \leq c'$ and hence $(b\op a) \op c$ is defined. Using (W1),  we get $(a \op b) \op c$ is defined as required. To show that $a \op (b \op c)= (a \op b) \op c$ notice that from $a \leq (b \op c)'$ and the fact that $b \op c= (b' \cdot c')'$ and $(a')'=a$ we get $(a')' \leq b' \cdot c'$. Similarly from $b \leq c'$ we get $(b')' \leq c'$. Then using (cw4) we get $a'\cdot (b'\cdot c')=(a'\cdot b')\cdot c'$ and so $(a'\cdot((b'\cdot c')')')'=(((a'\cdot b')')'\cdot c')'$. Simplifying both sides by applying (cw1) twice we get, $a' \rightarrow (b' \rightarrow c) = (a' \rightarrow b)' \rightarrow c$ which, by definition is $a \op (b \op c) = (a \op b) \op c$ as required. Finally, $(W3)$ follows from (cw5) as $a\op 0=(a'\cdot 1)'=(a')'=a$.
\end{proof}
\begin{corollary}
Let $(L;\leq,\cdot,\rightarrow,0,1)$ be an involutive CI-lattice that satisfies {\rm(cw1)}, {\rm(cw2)}, {\rm(cw3)} and {\rm(cw4)}. Then $L$ can be arranged into a lattice effect algbera $(L;\op,0,1)$ by setting $a'=a\rightarrow 0$, for all $a \in L$, $a\op b=a'\rightarrow b$ for all $a,b \in L$ with $a\leq b'$ and $a\rightarrow_s b=a\rightarrow b$.
\end{corollary}

\begin{proof}
First we should notice that from (CI1) we obtain (cw5) and so it is obtained that $(L;\leq,\op,',0,1)$ is a weak lattice effect algebra. From the fact that $a\rightarrow_s b=a'\op(a\wedge b)=a\rightarrow(a\wedge b)$, using (CI5) and (CI6), we conclude $a\rightarrow_s b=a\rightarrow b$ and  $a\leq b$ if and only if $a\rightarrow_s b=1$. The results will hence follow from Theorem \ref{RD}.
\end{proof}

\section{Logic of lattice effect algebra (LEL)} \label{sec4}
\noindent Let $P$ be a set of atomic propositions and $p\in P$. The set of formulas of LEL is inductively defined as:
$$\phi::=\ \ p\mid\perp\mid \phi\rightarrow \phi$$
\noindent In this setting we can define the standard logical connectives as
\begin{itemize}
\item Negation: $\neg \phi= \phi\rightarrow\perp$, 
\item Conjunction: $\phi\wedgedot \psi=\neg(\phi\rightarrow\neg \psi)$, 
\item Disjunction: $\phi\veedot \psi=\neg(\neg \phi\wedgedot\neg \psi)$ 
\item $\top=\neg\bot$.
\end{itemize}
 
\noindent With this definitions the connectives of our logic corresponds to those proposed in \cite{Foulis2012}.

\subsection{Semantics of LEL}
The semantics of the LEL will be given in a lattice effect algebra.

\begin{definition}
Let $E$ be a lattice effect algebra. A valuation is a function $V:P\rightarrow E$ that assigns to each atomic proposition an element of $E$ and extends to all formulas as follows
\begin{align*}
&V(\bot)=0; &&V(\top)=1; && V(\neg \phi)=V(\phi)';\\
&V(\phi \rightarrow \psi)&=&(V(\phi)\otimes V(\psi)')'&=& V(\phi)\rightarrow_s V(\psi);&\\
&V(\phi \wedgedot \psi)&=& V(\phi)\otimes V(\psi)&=& (V(\phi)\rightarrow_s V(\psi)')';&\\
&V(\phi \veedot \psi)&=& (V(\phi)'\otimes V(\psi)')'&=& V(\phi)'\rightarrow_s V(\psi).&\\
\end{align*}

\noindent Then we say that $\mathcal{E}=\langle E, V\rangle$ is a model of LEL.
\end{definition}

 \noindent We can also define two other logical connectives in our language to correspond to lattice operations of $E$ as $\phi \wedge \psi= \phi \wedgedot(\phi \rightarrow \psi)$ and $\phi \vee \psi=\neg(\neg \phi \wedge\neg \psi)$.

\begin{lemma}\label{SAwedge}
Let $E$ be a lattice effect algebra and $a,b,c\in E$. Then
$$a\wedge b=a\otimes(a\rightarrow_s b)=b\otimes(b\rightarrow_s a)=(a\rightarrow_s(a\rightarrow_s b)')'=(b\rightarrow_s(b\rightarrow_s a)')'$$
\end{lemma}

\begin{proof}
From (CI3), we have $a\otimes(a\rightarrow_s b)\leq a$ and from (CI4), we have $a\otimes(a\rightarrow_s b)\leq b$. Hence $a\otimes(a\rightarrow_s b)\leq a\wedge b$. Now, using the divisibility laws of the lattice effect algebras in Proposition \ref{LEAtoCI}, we obtain $a\wedge b=a\otimes(a\rightarrow_s b)=b\otimes(b\rightarrow_s a)$. Then, from the duality of $\otimes$ and $\rightarrow_s$, we have $a\wedge b= (a\rightarrow_s(a\rightarrow_s b)')'=(b\rightarrow_s(b\rightarrow_s a)')'$.
\end{proof}

\noindent The semantics for $ \wedge$ and $ \vee$ are then given by \\
$V(\phi \wedge \psi)=V(\phi)\otimes(V(\phi)\rightarrow_s V(\psi))=(V(\phi)\rightarrow_s (V(\phi)\rightarrow_s V(\psi))')'= V(\phi)\wedge V(\psi)$;\\
$V(\phi \vee \psi)=V(\phi)\vee V(\psi)$.\\

\noindent \textbf{Note.} If $V(\phi) \perp V(\psi)$ then $V(\phi\wedgedot \psi)=V(\phi)\od V(\psi)$ and $V(\phi \veedot \psi)=V(\phi)\op V(\psi)$.

\begin{definition}
We say a formula $\phi$ in LEL is valid in a model $\mathcal{E}=\langle E, V\rangle$ if $V(\phi)=1$.
\end{definition}
 
\subsection{An axiom system for LEL}
In this section, we define an axiom system for LEL. Along the same lines as  \cite{Pavicic1992}, where they define an axiom system for a logic of orthomodular lattices, we use axiom schemata instead of axioms and from now on whenever we mention axioms we mean axiom schemata. In the presentation of axioms and rules, $\vdash$, $\&$, $\Rightarrow$ and $\Leftrightarrow$ are to be understood as symbols in the metalanguage that are interpreted as ``it can be asserted in LEL'', and the usual classical connectives, ``AND", ``IF .. THEN..."  and ``IF AND ONLY IF" respectively and we will write $\vdash \phi\leftrightarrow \psi$ for $\vdash \phi\rightarrow \psi\ \&\ \vdash \psi\rightarrow \phi$.\\

\textbf{Axioms:}
\begin{itemize}
\addtolength{\itemindent}{1cm}
\item[(A1)]{\makebox[3cm] {$\vdash \phi\rightarrow \phi$; \hfill}
{(A2) $\vdash \phi\leftrightarrow\neg\neg \phi$; \hfill} (A3) $\vdash \phi\rightarrow\top$;}
\item[(A4)] {\makebox[3cm] {$\vdash \phi\wedge \psi\rightarrow \phi$;\hfill}
{(A5) $\vdash \phi\wedge \psi\rightarrow \psi$.\hfill}{} }
\end{itemize}

\textbf{Rules:}
\begin{itemize}
\addtolength{\itemindent}{3mm}
\item[(R1)] $\vdash \phi\quad\Leftrightarrow\quad\vdash \top\rightarrow \phi$;
\item[(R2)] $\vdash \phi\rightarrow \psi\quad\&\quad\vdash \psi\rightarrow \chi\quad\Rightarrow\quad\vdash \phi\rightarrow \chi$; 
\item[(R3)] $\vdash \phi\rightarrow \psi\quad\Rightarrow\quad\vdash\neg \psi\rightarrow\neg \phi$;
\item[(R4)] $\vdash \phi\rightarrow \psi\quad\Rightarrow\quad\vdash(\neg \phi\rightarrow \psi)\leftrightarrow(\neg \psi\rightarrow \phi)$;
\item[(R5)] $\vdash \phi\rightarrow \psi\quad\Rightarrow\quad\vdash \phi\rightarrow(\neg \phi\rightarrow \psi)$;
\item[(R6)] $\vdash \phi\rightarrow \psi\quad\Rightarrow\quad\vdash \chi\wedgedot \phi\rightarrow \chi\wedgedot \psi$;
\item[(R7)] $\vdash \phi\leftrightarrow \psi\quad\Rightarrow\quad\vdash \phi\wedgedot \chi\leftrightarrow \psi\wedgedot \chi$; 
\item[(R8)] $\vdash \phi\rightarrow \psi\quad\&\quad\vdash \phi\rightarrow \chi\quad\Rightarrow\quad\vdash \phi\rightarrow \psi\wedge \chi$;
\item[(R9)] $\vdash \phi\rightarrow\neg \psi\ \ \&\ \ \vdash \phi\rightarrow\neg \chi\ \ \&\ \ \vdash(\neg \phi\rightarrow \psi)\rightarrow\neg \chi\ \Rightarrow\ \vdash(\neg \phi\rightarrow \chi)\rightarrow\neg \psi$;
\item[(R10)] $\vdash \neg \psi\rightarrow \chi\ \ \&\ \ \vdash \neg \phi\rightarrow \psi\wedgedot \chi\ \ \Rightarrow\ \ \vdash \phi\wedgedot(\psi\wedgedot \chi)\leftrightarrow (\phi\wedgedot \psi)\wedgedot \chi$.
\end{itemize}
\vspace{2mm}

\begin{lemma}
The followings are provable in LEL:
\begin{itemize}
\addtolength{\itemindent}{1cm}
\item[{\rm(A6)}] {\makebox[5.5cm] {$\vdash \phi\leftrightarrow \phi\wedgedot\top$; \hfill} \rm(R11) $\vdash \phi\quad\Leftrightarrow\quad\vdash \phi\leftrightarrow\top$;}

\item[{\rm(R12)}]{\makebox[5.5cm] {$\vdash \phi\quad\&\quad\vdash \phi\rightarrow \psi\quad\Rightarrow\quad\vdash \psi$; \hfill}\rm(R13) $\vdash \phi\quad\Rightarrow\quad\vdash \psi\rightarrow \phi$;}
\item[{\rm(R14)}]{\makebox[5.5cm] {$\vdash \phi\rightarrow \psi\quad\Leftrightarrow\quad\vdash \phi\rightarrow \phi\wedge \psi$.\hfill}}
\end{itemize}
\end{lemma}

\begin{proof}
\noindent {\bf(A6)}: By definition of $\wedgedot$ we have $\phi\wedgedot\top=\neg(\phi\rightarrow\bot)=\neg(\neg \phi)$ then (A6) follows from (A2) and (R2).

\noindent {\bf(R11)}: This is clear from (A3) and (R1).

\noindent {\bf(R12)}: Let $\vdash \phi$ and $\vdash \phi\rightarrow \psi$. Then from (R1) we have $\vdash \top\rightarrow \phi$ and using (R2), we get $\vdash \top\rightarrow \psi$. Hence applying (R1) again we have $\vdash \psi$.

\noindent {\bf(R13)}: Let $\vdash \phi$. Then $\vdash \top\rightarrow \phi$ follows form (R1) and $\vdash \psi\rightarrow\top$ follows from (A3). Thus using (R2), we have $\vdash \psi\rightarrow \phi$.

\noindent {\bf(R14)}: Let $\vdash \phi\rightarrow \psi$. Then, from (A1) and (R8), we get $\vdash \phi\rightarrow \phi\wedge \psi$. Conversely, suppose $\vdash \phi\rightarrow \phi\wedge \psi$. Then, from (A5) and (R2), we have $\vdash \phi\rightarrow \psi$.\\
\end{proof}

\begin{theorem}{\rm(soundness)}
If $\vdash \phi$ then $\phi$ is a valid formula in every model of LEL.
\end{theorem}
\begin{proof}
\noindent First we should notice that, by Theorem \ref{RD} (RD4), $V(\phi \rightarrow \psi)=1$ if and only if $V(\phi)\leq V(\psi)$. The soundness of all axioms and rules will then follow easily except for (R9) and (R10). Notice then that (R9) and (R10) are the straightforward translation of (cw3) and (cw4), respectively.
\end{proof}

\begin{definition}
Consider the free generated algebra $\mathcal{A}=\langle P, \rightarrow,\bot\rangle$ in which $P$ is the set of atomic propositions and the operations satisfy the axioms and rules of LEL, we define the relation $\equiv$ on $\mathcal{A}$ as
$$\phi\equiv \psi\quad\text{iff}\quad \vdash \phi\leftrightarrow \psi.$$
\end{definition}

\begin{lemma}
The relation $\equiv$ is a congruence on $\mathcal{A}$.
\end{lemma}
\begin{proof}
It is easy to see that $\equiv$ an equivalence relation so we show that it preserves the operations. Let $\phi\equiv \psi$. Then $\vdash \phi\leftrightarrow \psi$ and from (R3) we have $\vdash \neg \phi\leftrightarrow \neg \psi$ which means $\neg \phi\equiv\neg \psi$. Now suppose that $\phi\equiv \psi$ and $\chi\equiv \theta$. From $\phi\equiv \psi$ by (R7) we get $\phi\wedgedot \chi\equiv \psi\wedgedot \chi$ and from $\chi\equiv \theta$ by (R6) we get $\psi\wedgedot \chi\equiv \psi\wedgedot \theta$. Then from (R2) we have $\phi\wedgedot \chi\equiv \psi\wedgedot \theta$ and similarly $\chi\wedgedot \phi\equiv \theta\wedgedot \psi$.
\end{proof}
\begin{proposition}
The Lindenbaum–Tarski algebra $\mathcal{A}/\!\!\equiv$ with the following operations is a lattice effect algebra.
\begin{align*}
0=&[\bot],\\
1=&[\top],\\
[\phi]'=&[\neg \phi],\\
[\phi] \cdot [\psi]=&[\phi\wedgedot \psi],\\
[\phi]\to[\psi]=&[\phi\rightarrow \psi].
\end{align*}
\end{proposition}
\begin{proof}
Let $[\phi],[\psi]\in \mathcal{A}/\!\!\equiv$ and define $[\phi]\leq[\psi]$ if and only if $\vdash \phi\rightarrow \psi$. Notice that from (A1) and (R2) this order is well defined.

\noindent First we show that $\mathcal{A}/\!\!\equiv$ is an involutive lattice. From (A4) and (A5) we get $[\phi\wedge \psi]\leq [\phi], [\psi]$ then from (R8) we have that $[\phi\wedge \psi]$ is the greatest lower bound of $[\phi]$ and $[\psi]$ that is $[\phi] \wedge [\psi]=[\phi\wedge \psi]$. Moreover by (A2) and (R3) the operation $'$ is an involution and thus $[\neg \phi\wedge\neg \psi]'$ will be the smallest upper bound of $[\phi]$ and $[\psi]$, i.e. $[\phi] \vee [\psi]$, and hence $\mathcal{A}/\!\equiv$ is an involutive lattice. 

\noindent Using the definition of $\wedgedot$, (R4), (R5), (R9), (R10) and (A6) we obtain (cw1), (cw2), (cw3), (cw4) and (cw5) and thus by Theorem \ref{invlatweak} $\mathcal{A}/\!\!\equiv$ is a weak lattice effect algebra in which $[\phi]\op[\psi]=[\phi]'\to [\psi]$ if and only if $[\phi]\leq [\psi]'$. Hence $[\phi]\rightarrow_s[\psi]=[\phi]'\op ([\phi]\wedge[\psi])=[\phi] \to ([\phi]\wedge[\psi])= [\phi] \to [\phi \wedge \psi]$. 
Now from definition of the order we have $[\phi]\leq[\psi]$ if and only if $\vdash \phi\rightarrow \psi$ if and only if $\vdash \phi\rightarrow \phi\wedge \psi$ (by (R12)) if and only if $[\phi]\rightarrow_s[\psi]=[\phi]\to [\phi \wedge \psi]=[\phi \rightarrow \phi\wedge \psi]=[\top]=1$ (by (R11)). Therefore by (RD4) in Theorem \ref{RD} we have that $\mathcal{A}/\!\equiv$ is a lattice effect algebra.
\end{proof}
\begin{corollary}\label{canonicalmodel}
The structure $\bar{\mathcal{E}}=\langle \mathcal{A}/\!\!\equiv, \bar{V}\rangle$ is a model of LEL in which the valuation $\bar{V}:P\rightarrow \mathcal{A}/\!\!\equiv$ is defined as $\bar{V}(p)=[p]$ for any atomic proposition $p$.
\end{corollary}
\begin{lemma}\label{complete}
For every formula $\phi$, $\vdash \phi$ if and only if $[\phi]=1$.
\end{lemma}
\begin{proof}
$\vdash \phi$ if and only if $\vdash \phi\leftrightarrow\top$ (by (R11)) if and only if $[\phi]=[\top]=1$, by the definition of the congruence. 
\end{proof}
\begin{theorem}{\rm(completeness)}
If $\phi$ is a valid formula in every model of LEL then $\vdash \phi$.
\end{theorem}
\begin{proof}
From Corollary \ref{canonicalmodel}, $\bar{\mathcal{E}}$ is a model of LEL and by the assumption $\phi$ is a valid formula in it such that $\bar{V}(\phi)=[\phi]=1$. Therefore by Lemma \ref{complete}, $\vdash \phi$.
\end{proof}
\section{Conclusion}

As suggested in \cite{Foulis2012} lattice effect algebras can be enriched with natural binary operations of Sasaki product and Sasaki arrow that show a similar characteristic as the well studied Sasaki projection and Sasaki hook. Using this binary operations to represent conjunction and implication, the lattice effect algebras can be arranged into a conjunction/implication lattice and can be studied as a semantic model for a logical calculus in the same manner that orthomodular lattices and MV-algebras can act as the semantic models for sharp quantum logic and Lukasiewicz many valued logics.  

We studied the logical structure of lattice effect algebras enriched with these operations. We characterised the adjointness between the Sasaki product and Sasaki arrow on the lattice effect algebras in a manner analogues to the characterisation of the duality between Sasaki projection and Sasaki hook in terms of the orthomodular law on ortholattices. We then studied the logical structure of lattice effect algebras by presenting a sound and complete axiomatisation for a logical calculus for which the lattice effect algebras can act as a semantic model. 


\end{document}